\documentclass[12pt]{amsart}
\usepackage{latexsym, amsbsy, amsmath, amsfonts, amssymb, amsthm, amscd, amscd}
\usepackage[french,english]{babel}
\usepackage[latin1]{inputenc}
\usepackage[all]{xy}
\newcommand{\Z}{\mathbb Z}

\newcommand{\Q}{\mathbb Q}
\newtheorem{Theorem}{Theorem}[section]

\allowdisplaybreaks[2]
\numberwithin{equation}{section}
\numberwithin{figure}{section}
\title{On the Lie algebras of surface pure braid groups}

\author[Enriquez]{B.~Enriquez}
\address{IRMA (CNRS), rue Ren\'e Descartes, F-67084 Strasbourg, France}
\email{enriquez@@math.u-strasbg.fr}
\author[Vershinin]{V.~V.~Vershinin}
\address{D\'epartement des Sciences Math\'ematiques,
                                     Universit\'e Montpellier II,
Place Eug\`ene Bataillon,
34095 Montpellier cedex 5, France}
\email{ vershini@math.univ-montp2.fr}
\address{Sobolev Institute of Mathematics, Novosibirsk 630090,
Russia } 
\email{ versh@math.nsc.ru}


\begin{document}
\begin{abstract}
We consider the Lie algebra associated with the
descending central series filtration of the pure braid group of a closed surface
of arbitrary genus.
R.~Bezrukavnikov gave a presentation of this Lie algebra over the rational 
numbers. We show that his presentation remains true for this
Lie algebra itself, i.\,e. over integers.
\end{abstract}
\maketitle
\tableofcontents

\section{Introduction}

Let $S_{g,b}$ be an oriented surface of genus $g$ with $b$ boundary components.
We denote by $Br_n(S_{g,b})$ the braid group on $n$ strings of the surface 
$S_{g,b}$
and by $P_n(S_{g,b})$ the pure braid group  on $n$ strings 
of the surface $S_{g,b}$.
There exists an exact sequence:
\begin{equation*}
1 \to {P}_n(S_{g,b}) \to {Br}_n(S_{g,b}) \to \Sigma_{n}\to 1,
\end{equation*}
where ${Br}_n(S_{g,b}) \to \Sigma_{n}$ is a natural epimorphism to the 
symmetric group  
$\Sigma_{n}$, so, $P_n(S_{g,b})$ is its kernel.
The classical braid groups are the braid groups of a disc: 
$Br_n \cong Br_n(S_{0,1})$,
$P_n \cong P_n(S_{0,1})$. 

In the paper we consider the pure braid group of a  closed orientable 
surface of genus $g$ with 
no boundary components, ${P}_n(S_{g,0})$, which we denote for
simplicity by ${P}_n(S_{g,})$ 

Usually  the braid group $Br_n$ 
is given by the following Artin presentation \cite{Art1}.
It has the generators $\sigma_i$, 
$i=1, ..., n-1$, and the two types of relations: 
\begin{equation}
 \begin{cases} \sigma_i \sigma_j &=\sigma_j \, \sigma_i, \ \
\text{if} \ \ |i-j|
>1,
\\ \sigma_i \sigma_{i+1} \sigma_i &= \sigma_{i+1} \sigma_i \sigma_{i+1} \, .
\end{cases} \label{eq:brelations}
\end{equation}

The generators $a_{i,j}$, $1\leq i<j\leq n $ for the pure
braid group $P_n$ (of a disc) can be defined (as elements of the the braid 
group $Br_n$)  by the formula:
$$a_{i,j}=\sigma_{j-1}...\sigma_{i+1}\sigma_{i}^2\sigma_{i+1}^{-1}...
\sigma_{j-1}^{-1}.$$ 
Then the defining relations, which are called the \emph{Burau
relations} \cite{Bu1}, \cite{Mar2} are as follows: 
\begin{equation}
\begin{cases}
a_{i,j}a_{k,l}=a_{k,l}a_{i,j}
\ \text {for} \ i<j<k<l \ \text {and} \ i<k<l<j, \\
a_{i,j}a_{i,k}a_{j,k}=a_{i,k}a_{j,k}a_{i,j} \ \text {for} \
i<j<k, \\
a_{i,k}a_{j,k}a_{i,j}=a_{j,k}a_{i,j}a_{i,k} \ \text
{for} \ i<j<k, \\
a_{i,k}a_{j,k}a_{j,l}a_{j,k}^{-1}=a_{j,k}a_{j,l}a_{j,k}^{-1}a_{i,k}
\ \text {for} \ i<j<k<l.\\
\end{cases}
\label{eq:burau}
\end{equation}

\smallskip

\noindent
It was proved by O.~Zariski \cite{Za1} and then
rediscovered by E.~Fadell and J.~Van Buskirk \cite{FaV} that a presentation 
for the braid group
of a sphere can be given with the   generators 
$\sigma_i$, $i=1, ..., n-1$, the same as for the classical braid group, satisfying
the braid relations (\ref{eq:brelations})
and the following sphere relation: 
\begin{equation*}
\sigma_1 \sigma_2 \dots \sigma_{n-2}\sigma_{n-1}^2\sigma_{n-2} \dots
\sigma_2\sigma_1 =1.
\end{equation*}
 
For the pure braid group on a sphere let us introduce the elements
$a_{i,j}$ for all $i, j$ by the formulae:  
\begin{equation}
\begin{cases}
a_{j,i}= a_{i,j} \ \ \text{for} \ i<j\leq n,\\
a_{i,i}= 1. 
\label{eq:aji}
\end{cases}
\end{equation}
The pure braid group for the sphere has the generators $a_{i,j}$
which satisfy Burau relations (\ref{eq:burau}), relations (\ref{eq:aji}),
 and the following relations \cite{GVB}:
\begin{equation*}
a_{i,i+1}a_{i,i+2} \dots a_{i,i+n-1} = 1 \ \ \text{for all} \ i\leq n,\\
\end{equation*}
where  $k+n$ is considered modulo $n$ having in mind (\ref{eq:aji}).

For a group $G$ the descending central series 
\begin{equation*}
G =\Gamma_1  > \Gamma_2 > \dots  > \Gamma_i > \Gamma_{i+1} > \dots .
\end{equation*}
\noindent
is defined by the formulas
\begin{equation*}
\Gamma_1 = G, \ \ \Gamma_{i+1} =[\Gamma_i, G].
\end{equation*}
It gives rise to the 
associated graded Lie algebra (over $\Z$) $gr^*(G)$ \cite{Serr}:
\begin{equation*}
gr^i(G)= \Gamma_i/\Gamma_{i+1}.
\end{equation*}

 A presentation of the Lie algebra $gr^*(P_n)$ for the pure braid group
can be described as follows \cite{K}. It is the quotient of the
free Lie algebra $L[A_{i,j}| \, 1 \leq i < j \leq n]$, generated by
elements $A_{i,j}$ with $1 \leq i < j \leq n$, modulo the
``infinitesimal braid relations" or ``horizontal $4T$ relations"
given by the following three relations:

\begin{equation*}
\begin{cases}
 [A_{i,j}, A_{s,t}] = 0, \  \text{if} \ \{i,j\} \cap \{s,t\} = \phi, \\
 [A_{i,j}, A_{i,k} + A_{j,k}] = 0, \  \text{if} \ i<j<k , \\
  [A_{i,k}, A_{i,j} + A_{j,k}] = 0, \ \text{if} \ i<j<k. \\
  \end{cases}
\end{equation*}

\smallskip

A similar presentation of the Lie algebra of descending central series 
filtration of the upper triangular McCool group is obtained in \cite{CPVW}.

Y.~Ihara in \cite{Ih} gave  a presentation of the Lie algebra $gr^*(P_n(S^2))$ 
of the pure braid group of the sphere. It is convenient to 
have conventions like (\ref{eq:aji}). Hence, $gr^*(P_n(S^2))$ 
is the quotient of the
free Lie algebra $L[B_{i,j}| \, 1 \leq i ,j \leq n]$ generated by
elements $B_{i,j}$ with $1 \leq i , j \leq n$ modulo the
 following  relations:
\begin{equation}
\begin{cases}
B_{i,j} =  B_{j,i} \  \text{for} \ 1\leq i,j \leq n , \\
B_{i,i} =  0 \ \ \text{for} \ 1\leq i \leq n , \\
 [B_{i,j}, B_{s,t}] = 0, \  \text{if} \ \{i,j\} \cap \{s,t\} = \phi, \\
 \sum_{j=1}^n B_{i,j} = 0, \  \text{for} \ 1\leq i \leq n. \\
    \end{cases}
    \label{eq:ihara}
\end{equation}
Also a presentation of the Lie algebra $gr^*(P_n(S^2))$ can be given with 
generators
$A_{i,j}$ with $1 \leq i< j  \leq n-1$,  modulo 
the following  relations \cite{KV}:
\begin{equation*}
\begin{cases}
 [A_{i,j}, A_{s,t}] = 0, \  \text{if} \ \ \{i,j\} \cap \{s,t\} = \phi, \\
 \smallskip
2  (\sum_{i=1}^{n-2}\sum_{j=i+1}^{n-1} A_{i,j}) = 0. 
 \end{cases}
\end{equation*}
So, the element $\sum_{i=1}^{n-2}\sum_{j=i+1}^{n-1} A_{i,j}$ of order 2 generates
the central subalgebra in  $gr^*(P_n(S^2))$. 

 We study the   natural Lie algebra
obtained from the descending central series for $(P_n(S_{g})$ for $g\geq 1$,
for $g=0$ the structure of  $gr^*(P_n(S^2))$ is given by the result of 
Y.~Ihara (\ref{eq:ihara}). 

The essential ingredient in the study of the pure braid groups is a natural
fibration of configuration spaces and its initial term of the homotopy exact
 sequence.
Let $F(S_g,n )$ be the space of $n$-tuples of pairwise different points 
in $S_{g} $, then we have the following fibration:
\begin{equation*}
S_{g}\setminus Q_{n-1} \to F(S_{g},n)\to F(S_{g},n-1 ),
\end{equation*}
where $S_{g}\setminus Q_{n-1}$ is the surface with $n-1$ points
deleted and $ F(S_{g},n)\to F(S_{g},n-1 )$ is the projection on the
first $n-1$ components of an $n$-tuple; $S_{g}\setminus Q_{n-1}$ is a fiber 
of the fibration.
This fibration generates the exact sequence of groups:
\begin{equation}
1 \to {\pi}_1(S_{g}\setminus Q_{n-1})\to {P}_n(S_{g})\to {P}_{n-1}(S_{g})\to 1,
\label{eq:ex_s_p}
\end{equation}
where ${\pi}_1(S_{g}\setminus Q_{n-1})$ is a $(g, n-1)$-{\it surface group}; 
it is a free group on $2g+n-2$ generators and it has the canonical presentation
\begin{equation*}
\pi_{g,n-1}={\pi}_1(S_{g}\setminus Q_{n-1})= <a_1, c_1, \dots, a_g, c_g, u_1, 
\dots, u_{n-1} \
 | \ \prod_{i=1}^{n-1}u_i  \prod_{m=1}^{g}[a_m,c_m]=1>.  
 \end{equation*}
If we consider the descending central series filtration of the groups of 
this exact sequence and apply the functor of the associated Lie algebras, then the 
corresponding sequence will be not left exact for $n\geq 3$:
\begin{equation*}
 gr^*({\pi}_1(S_{g}\setminus Q_{n-1}))\to gr^*({P}_n(S_{g}))\to 
 gr^*({P}_{n-1}(S_{g}))\to 1.
\end{equation*}
To fix the situation another filtration was introduced by M.~Kaneko
\cite{Kan} and H.~Nakamura and H.~Tsunogai \cite{NT}. The authors call
it the {\it{weight filtration}}. Roughly speaking the difference with
respect to the descending central series filtration is that the elements
 $ u_1, \dots, u_{n-1}$ are given the grading two instead of one. 
For $\pi_{g,0}$ and $\pi_{g,1} $ it coincides with the descending 
central series filtration.
\begin{equation*}
\pi_{g,k}(1) = \pi_{g,k}, \\
\end{equation*}
\begin{equation*}
\pi_{g,k}(2) = [\pi_{g,k}, \pi_{g,k}]<u_1, \dots, u_k >, 
\\
\end{equation*}
\begin{equation*}
\pi_{g,k}(m) = <[\pi_{g,k}(i), \pi_{g,k}(j)] \ | \, i+j=m >, m\geq 3. 
\end{equation*}

$${P}_n(S_{g})(1) = {P}_n(S_{g}), \\$$
$$
{P}_n(S_{g})(2) = [{P}_n(S_{g}), {P}_n(S_{g})] \, \pi_{g,n-1}(2), 
\\$$
$$
{P}_n(S_{g})(m) = <[{P}_n(S_{g})(i), {P}_n(S_{g})(j)] \ | \, i+j=m >, m\geq 3. 
$$

It is proved by H.~Nakamura, N.~Takao and R.~Ueno \cite{NTU} that the 
sequence of Lie algebras associated to the weight
filtration and corresponding to the sequence (\ref{eq:ex_s_p})

\begin{equation*}
1\to gr^*_w({\pi}_1(S_{g}\setminus Q_{n-1}))\to gr^*_w({P}_n(S_{g}))\to 
 gr^*_w({P}_{n-1}(S_{g}))\to 1
\end{equation*}
is exact (even for a punctured surface).

\section{Lie algebra $gr^*(P_n(S_{g}))$ 
}

R.~Bezrukavnikov \cite{Bez} determined the Lie algebra 
$gr^*(P_n(S_{g}))\otimes \Q$. It has the generators $a_{l,i}$, $b_{l,i}$
$1\leq l \leq g$, $1\leq i \leq n$, so $l$ corresponds to the genus of a surface
and $ i$ corresponds to the number of strings. Relations are as follows
\begin{multline}
\begin{cases}
[a_{l,i}, b_{k,j}]&=0 \ \ \text{for} \ i\not=j, \ l\not=k, \\
[a_{l,i}, a_{k,j}]&=0 \ \ \text{for} \ i\not=j,  \\
[b_{l,i}, b_{k,j}]&=0 \ \ \text{for} \ i\not=j,  \\
[a_{l,i}, b_{l,j}]&=[a_{k,j}, b_{k,i}] \ \ \text{for all } k, l \ 
\text{and} \ i\not=j,  \  \text{denote it by } \ s_{i,j},  \\
\sum_{l=1}^g[a_{l,i}, b_{l,i}]&= -\sum_{j\not=i} s_{i,j}, \\
[a_{l,i}, s_{j,k}]&=0 \ \ \text{for} \ i\not=j\not=k, \\
[b_{l,i}, s_{j,k}]&=0 \ \ \text{for} \ i\not=j\not=k. \\
\end{cases}
\label{eq:bez}
\end{multline}
 
The aim of the present work is to prove that this presentation is true over 
$\Z$. 

The Lie algebra  associated with the weight filtration for the pure braid group 
of a punctured surface was described by  H.~Nakamura, N.~Takao and R.~Ueno
 in \cite{NTU}.
For a surface without 
punctures the description of H.~Nakamura, N.~Takao and R.~Ueno
coincides with that of R.~Bezrukavnikov.

\begin{Theorem}  For the pure braid group of a closed surface of the genus $g$, $gr^*({P}_n(S_{g}))$, the descending central series
filtration and the weight filtration coincide. 
The generators $a_{l,i}$, $b_{l,i}$
$1\leq l \leq g$, $1\leq i \leq n$, and the Bezrukavnikov relations 
(\ref{eq:bez}) give a presentation of the graded Lie algebra 
$gr^*(P_n(S_{g}))$.
\end{Theorem} 
\begin{proof}
We denote by $\Gamma_i$ the descending central series
filtration on ${P}_n(S_{g})$. 
Let us consider the presentation of a surface pure braid group  given
by D.~L.~Gonçalves and J.~Guaschi \cite{GG}. There are $2ng$ generators:
 $\rho_{i, l}$ and $\tau_{i, l}$, $1\leq l \leq g$, $1\leq i \leq n$,
 and relations of 30 types . 
\par
Evidently, we have
$\Gamma_1= {P}_n(S_{g})(1)$, $\Gamma_2\subset {P}_n(S_{g})(2)$. So, there is 
a commutative diagram of the exact sequences
\begin{equation*}
\begin{CD}
1 \ \ \to \  \Gamma_2 @>>> {P}_n(S_{g})@>>>g^1({P}_n(S_{g})) \to 1  \\
 \ \ \ \ \ \ \  @V\psi VV   @VV Id V @V\phi VV\\
1\to {P}_n(S_{g})(2)@>>>{P}_n(S_{g})@>>> g^1_w({P}_n(S_{g})) \to 1
 \end{CD} 
\end{equation*}
The map $\phi$ is surjective and the rank of $g^1_w({P}_n(S_{g}))$ is 
equal to $2gn$. The rank of $g^1({P}_n(S_{g}))$ can not be bigger than
$2gn$, so it is an isomorphism. Hence $\psi$ is also an isomorphism.
From the definition of ${P}_n(S_{g})(m) $ it follows that
$\Gamma_m = {P}_n(S_{g})(m) $.

\end{proof}

The second author is thankful to John Guaschi for useful remarks.


\begin{thebibliography}{References}
\bibitem{Art1}
\emph{ E. Artin},
 Theorie der Z\"opfe.  Abh. Math. Semin. Univ. Hamburg, 1925, v. 4, 47--72.
\bibitem{Bez}
\emph{R.~Bezrukavnikov}, Koszul DG-algebras arising from configuration spaces.  
Geom. Funct. Anal.  4  (1994),  no. 2, 119--135.
 \bibitem{Bu1}
\emph{W.~Burau}, 
\"Uber Zopfinvarianten. (German)
Abh. Math. Semin. Hamb. Univ. 9, 117-124 (1932). 
\bibitem{CPVW} \emph{F.~R.~Cohen;  J.~Pakianathan; V.~V.~Vershinin; J.~Wu}, 
Basis-conjugating automorphisms of a free group and associated Lie algebras. 
Iwase, Norio (ed.) et al., Proc. of the conference on groups, homotopy and configuration spaces, Univ. of Tokyo, July, 2005. Coventry: Geometry and Topology Monographs 13, 147-168 (2008).
\bibitem{FaV}
\emph{E.~Fadell and J.~Van Buskirk}, The braid groups of $E\sp{2}$ and $S\sp{2}$. 
Duke Math. J. 29 1962,    243--257.
\bibitem{FR} \emph{M.~Falk, and R.~Randell}, The lower central series of a
fiber--type arrangement, Invent. Math., {\bf{82}} (1985), 77--88.
\bibitem{GVB}
\emph{R.~Gillette; J.~Van Buskirk}, The word problem and consequences for 
the braid groups and mapping class groups of the $2$-sphere.  
Trans. Amer. Math. Soc.  131.  1968. 277--296. 
\bibitem{GG}
\emph{D.~L.~Gonçalves; J.~Guaschi}, On the structure of the  surface pure
braid groups. J. Pure Appl. Algebra. 182 (2003), 33--64. 
\bibitem{Ih}
 \emph{Y.~Ihara},   Galois group and some arithmetic functions,
Proceedings of the International Congress of mathematicians, Kyoto,
1990, Springer (1991),  99-120. 
\bibitem{Kan} 
\emph{M.~Kaneko}, Certain automorphism groups of pro-$l$ fundamental groups of punctured Riemann surfaces.  J. Fac. Sci. Univ. Tokyo Sect. IA Math.  36  (1989),  no. 2, 363--372.
\bibitem{KV} 
\emph{R.~Karoui, V.~V.~Vershinin} On the Lie algebras associated with pure mapping class groups. preprint. 
\bibitem{K} \emph{T.~Kohno}, S\'erie de Poincar\'e-Koszul associ\'ee aux
groupes de tresses pure, Invent. Math., {\bf{82}} (1985), 57-75.
\bibitem{Mar2}
\emph{A. A. Markoff},  Foundations of the Algebraic Theory of
Tresses, Trudy Mat. Inst. Steklova, No~16, 1945 (Russian, English
summary).
\bibitem{NTU}
\emph{H.~Nakamura; N.~Takao; R.~Ueno}, Some stability properties of Teichmüller modular function fields with pro-$l$ weight structures.  
Math. Ann.  302  (1995),  no. 2, 197--213.
\bibitem{NT}
\emph{H.~Nakamura; H.~Tsunogai}, Some finiteness theorems on Galois centralizers in pro-$l$ mapping class groups.  J. Reine Angew. Math.  441  (1993), 115--144. 
\bibitem{Serr}
\emph{J.-P. Serre},
 Lie algebras and Lie groups. 1964 lectures given at Harvard University. Corrected fifth printing of the second (1992) edition. Lecture Notes in Mathematics, 1500. Springer-Verlag, Berlin, 2006. viii+168 pp.
\bibitem{Za1}
\emph{O.~Zariski}, On the Poincare group of rational plane curves, Am. J. Math.
 1936, 58,  607-619. 
\end{thebibliography}
\end{document}